\documentclass[11pt]{article}
\usepackage{amsmath}
\usepackage{amssymb,amsfonts}
\usepackage{amsthm}
\usepackage{cite}

\newtheorem{theorem}{Theorem}

{ \theoremstyle{remark}
}

{\theoremstyle{definition}
}
\newtheorem{lemma}{Lemma}[section]

\begin{document}

\title{\bf Classification of Realizations of Lie Algebras\\ of Vector Fields on a Circle}

\author{Stanislav Spichak}
\date{}
\maketitle

\begin{center}
Institute of Mathematics of NAS of Ukraine,\\
3 Tereshchenkivska Str., Kyiv-4, Ukraine\\
e-mail: stas.math@gmail.com
\end{center}

\begin{abstract}
Finite-dimensional subalgebras of a Lie algebra of
smooth vector fields on a~circle, as well as
piecewise-smooth global transformations of a~circle on itself, are considered. A
canonical forms of realizations of two- and three-dimensional noncommutative
algebras are obtained. It is shown that all other realizations
of smooth vector fields are reduced to this form using global
transformations. Some combinatorial formulas for the number of inequivalent realizations of these algebras
are obtained.
\end{abstract}

\section{Introduction}
The description of Lie algebra representations by vector fields on a line and a~plane  was
first considered by S.~Lie %\cite{spichak:Lie1},
\cite[S.~1--121]{spichak:Lie2}. However, this problem is still of great interest and widely applicable.
In spite to its importance for applications, only recently a complete description of
realizations begun to be investigated systema\-ti\-cally.
Furthermore, only since the late eighties of the last century papers on that subject were published regularly.
In particular, different problems of realizations were studied such as realizations of first order differential operators of a special form  in \cite{spichak:Gonzalez}, realizations of physical algebras (Galilei, Poincar\'e and Euclid ones)
 in \cite{spichak:Yehorchenko,spichak:Zhdanov2}. In \cite{spichak:Popovych} it was constructed
a complete set of inequivalent realizations of real Lie algebras of dimension no greater than four in vector fields on a space of
an arbitrary (finite) number of variables. In that paper one can obtain a more complete review on the subject and a list of references.

Almost in all works on the subject realizations are considered up to local equivalence transformations.
Attempts to classify realizations  of Lie algebras in vector fields on some manifold with respect to
global equivalence transformations (on the whole manifold) have been made only in a few papers (see, e.g.,
\cite{spichak:Zaitseva,spichak:Strigunova,spichak:Sergeev}).
In these papers it is proved that (up to isomorphism) there are only three algebras, namely,
one-dimensional, noncommutative two-dimensional and three-dimensional isomorphic to $\mathrm{sl}(2,\mathbb R)$,
that can be realized by analytic vector fields on the circle.

The purpose of this paper is to construct all inequivalent realization of the two-dimensional algebras on the circle.
For this reason, we are not limited by the requirement of analyticity, as it is considered in \cite{spichak:Zaitseva,spichak:Strigunova,spichak:Sergeev}.

On the circle $S^1$ we introduce the parameter $\theta\in {\mathbb
R}$, $0\leqslant\theta < 2\pi$. Then the vector fields on $S^1$ can
be  represented as a vector field $v(\theta)\frac{d}{d\theta}$,
where $v(\theta)$ is a smooth real function on the circle
\cite{spichak:Presli}. One more possibility is to assume that
$\theta\in\mathbb R$, and $v(\theta)$  is a smooth $2\pi$-periodic
function on a line. We consider the vector fields of the class $C^1$
(with continuously-differentiable functions $v(\theta)$), which is
natural to require when calculate the commutators of two vector
fields.

We also introduce a class of transformations $f\colon S^1\rightarrow S^1$, which is defined by the following properties:
\begin{list}{--}{\labelwidth=3ex\labelsep=1ex\leftmargin=5.5ex
\topsep1mm\parsep0mm\itemsep1mm\partopsep0mm}

\item it is one-to-one mapping of the circle onto itself;
\item $f(\theta)$ is continuous at any point $\theta\in S^1$;
\item it is continuously differentiable at all points except a finite number of them;
\item the derivative $f'(\theta)$ tends to $-\infty$ or $+\infty$ at all points of discontinuity;
\item under a change of the coordinate $\tilde{\theta}=f(\theta)$ a vector field of the class $C^1$ transforms to a vector field of the same class.
\end{list}

Such transformations are defined as {\it equivalence transformations} of vector fields. We call two realizations of an algebra of vector fields inequivalent,
if it is impossible to transform realizations to each other by compositions of equivalence transformations.

We assume that $f(0)=0$, without loss of generality. Indeed, any equivalence transformation is obvious a composition of
some equivalence transformations with fixed zero and a rotation of the circle. So, we can perform preliminary classification of  inequivalent realizations
up to such transformations (with fixed zero) and then complete classification taking into account the rotations of the circle. Then it is easy to see that
$f(\theta)$ is monotone on the whole interval
$0\leqslant\theta<2\pi$. Degree of the map equals $\pm1$, $\deg f=\pm 1$ (see \cite{spichak:Dubrovin}). Depending on the sign of the degree
the function is monotonically decreasing or increasing.

Besides the transformations of $f$, we introduce the following class of {\it homotopy} of a~circle into itself. Let us take two points $\theta_1,\theta_2\in S^1$,
$\theta_1<\theta_2$. We define a~family of transformations  $F_{\theta_1,\theta_2}(t,\theta)\colon [0,1]\times S^1\rightarrow S^1$ of the circle the as follows.

In the case $\theta_1\not=0$
$$
F_{\theta_1,\theta_2}(t,\theta)=\left\{
    \begin     {array}{clc}
   \displaystyle  \theta+t\dfrac{\theta}{\theta_1}(\theta_2-\theta_1)\quad & \text{if}\ \ 0\leq\theta <\theta_1 ,\\[1ex]
     \theta+t(\theta_2-\theta) & \text{if}\ \ \theta_1\leq\theta <\theta_2 ,\\[1ex]
    \theta & \text{if}\ \ \theta_2\leq\theta < 2\pi .
    \end {array}
    \right .
$$

In the case $\theta_1=0$
$$
F_{\theta_1,\theta_2}(t,\theta)=\left\{
    \begin     {array}{clc}
     \theta(1-t) & \text{if}\ \ 0\leq\theta <\theta_2 ,\\[1ex]
   \displaystyle  \theta-t\frac{2\pi-\theta}{2\pi-\theta_2}\theta_2\quad & \text{if}\ \ \theta_2\leq\theta < 2\pi .
    \end {array}
    \right .
$$

%so that $F_{\theta_1,\theta_2}(0,\theta)=\theta$, $F_{\theta_1,\theta_2}(1,\theta')=F_{\theta_1,\theta_2}(1,\theta_1)$,
%$\theta_1\leqslant\theta'\leqslant\theta_2$.

Obviously, if some subset of singular points of a vector field (i.e., zeros of its coefficient \cite{spichak:Dubrovin}) forms an open set
$(\theta_1,\theta_2)$ on $S^1$ then this interval can be constricted to a point by an appropriate homotopy.
Therefore, we can consider vector fields singularities  which do not form intervals, although the number of singularities may be infinite.
We denote the class of such vector fields by $\mathcal{C}$. They form the set of realizations of the one-dimensional algebra.
In general case such vector fields cannot be simplified by equivalence
transformations, because of infinity  the number of singular points.

We are interested in all inequivalent realizations of
finite-dimensional Lie algebras by vector fields from the class $\mathcal{C}$.

\section{Two-dimensional commutative algebra}
Further we denote vector fields $v(\theta)\frac{d}{d\theta}$ and
$w(\theta)\frac{d}{d\theta}$ by $V$ and $W$, correspondingly.  Let
they commute, and $V\in \mathcal{C}$. A singular point $\theta_0$ of
the field $V$ is said to be {\it degenerate} if $v'(\theta_0)=0$. It
is easily to show that if $0\leqslant\theta_0<\theta_1 < 2\pi$ are
two degenerate points, such that the interval $(\theta_0,\theta_1)$
does not contain additional degenerate points, then
$w(\theta)=\lambda v(\theta)$ on $(\theta_0,\theta_1)$, where
$\lambda\not=0$ is an arbitrary constant. This follows from the fact
that $W\in \mathcal{C}$ and the continuity of the derivative of the
function $w(\theta)$. In particular, if there is no degenerate
point, or it is unique, then $w(\theta)=\lambda v(\theta)$ on $S^1$.
 In this case the vector fields $V$ and $W$ are linearly dependent, and there is no realization of two-dimensional commutative Lie algebra (see \cite{spichak:Zaitseva,spichak:Strigunova,spichak:Sergeev}).

We assume that we have more than one degenerate point. Without loss
of generality, we may assume that point $0$ (and $2\pi$) is
degenerate. Then the function~$w(\theta)$ can be described as
follows. We take an arbitrary point $\theta\in S^1$. If it is
non-degenerate for function $v(\theta)$, then, obviously, there is a
maximum interval $(\theta_0,\theta_1)$ with two degenerate endpoints
on it, such that $0\leqslant\theta_0<\theta<\theta_1 < 2\pi$. Then
$w(\theta)=\lambda v(\theta)$ on this interval. Further, considering
point $\theta'\not\in[\theta_0,\theta_1]$, we repeat the procedure,
if it is not degenerate. Again we have relation
$w(\theta)=\lambda'v(\theta)$  on some interval. Here the values
$\lambda$ and $\lambda'$ can  be not equal.

If the point $\theta'$ is degenerate, then, by virtue of the fact that $V\in \mathcal{C}$, we can find arbitrarily close to it a non-degenerate point $\theta''$.
So we repeat the above procedure for $\theta''$.  Thus if we know the  function $v(\theta)$, we can construct values of the function~$w(\theta)$ in all
non-degenerate points. For degenerate points we can construct values of $w(\theta)$ in arbitrary close to them points.
If the number of degenerate points of the function $v(\theta)$ is infinite, then the number of non-equivalent realizations of the
vector field $W$ is infinite. So the number of realizations of two-dimensional commutative algebra is infinite.

\section{Two-dimensional noncommutative algebra.\\ Auxiliary lemmas}

Let vector fields $V$ and $W$ generate the noncommutative algebra.
It can be assumed up to their linear combining that they satisfy the commutation relation $[V,W]=W$,
which implies the relation between the functions $v(\theta)$ and~$w(\theta)$:
\begin{equation} \label{spichak:eq1}
v(\theta)w'(\theta)-v'(\theta)w(\theta)=w(\theta).
\end{equation}

\begin{lemma}\label{spichak:lemma1}
There is a singular point for the field $W$.
\end{lemma}

\begin{proof}
Assume that the field $W$ has no singular points, i.e.,
$w(\theta)>0$ (or $w(\theta)<0$) for all values $0\leqslant\theta <
2\pi$. Then from (\ref{spichak:eq1}) we can obtain the solution for
the function~$v(\theta)$ on the whole interval $[0,2\pi)$:
\begin{equation} \label{spichak:eq2}
v(\theta)= \left(-\int_0^\theta\dfrac{d\vartheta}{w(\vartheta)}+\lambda\right)w(\theta),
\end{equation}
where $\lambda$ is some constant.
The function $w(\theta)$  is $2\pi$-periodic on the line.
Since the integrand in (\ref{spichak:eq2}) is positive, we have $v(0)\not=v(2\pi)$.
It contradicts the periodicity of the function $v(\theta)$.
\end{proof}

\begin{lemma}\label{spichak:lemma2}
Singular points of $W$ are singular points of $V$.
\end{lemma}

\begin{proof}
Assume that $w(\theta_0)=0$. Suppose that $v(\theta_0)\not=0$. Then there exists some neighborhood $U_{\theta_0}$
of this point, where $v(\theta)\not=0$. In this neighborhood, the equation (\ref{spichak:eq1}) can be rewritten as:
\begin{equation} \label{spichak:eq3}
w'(\theta)=\dfrac{1+v'(\theta)}{v(\theta)}\,w(\theta).
\end{equation}
Since $w(\theta_0)=0$ and the right-hand side of equation (\ref{spichak:eq3}) satisfies the Lipschitz condition with respect to $w$
uniformly on $\theta$ then, by virtue of Picard's theorem \cite{spichak:Stepanov}, the differential equation (\ref{spichak:eq3}) has a unique solution in the neighborhood of
$\theta_0$. Obviously, such a solution is $w\equiv 0$, what contradicts the fact that $W\in \mathcal{C}$.
\end{proof}

\begin{lemma}\label{spichak:lemma3}
The number of singular points of the field $W$ is finite.
\end{lemma}

\begin{proof}
Assume that the number of singular points is infinite. Since $S^1$ is a compact set then there is a monotonically
increasing (or decreasing) a sequence~$\{\theta_n\}$ converging to some point $\theta_0$ such that $w(\theta_n)=0$. It is easy to show that
for any $n$ there is a non-singular point $\hat\theta_n\in (\theta_n,\theta_{n+1})$ satisfying the condition $w'(\hat\theta_n)=0$.
Then, from equation (\ref{spichak:eq1}) we see that $v'(\hat\theta_n)=-1$. Since $\displaystyle\lim_{n\rightarrow\infty}\hat\theta_n=\theta_0$,
then, by virtue of continuous differentiability of function $v(\theta)$, we have $\displaystyle\lim_{n\rightarrow\infty}v'(\hat\theta_n)=v'(\theta_0)=-1$.
On the other hand, from Lemma \ref{spichak:lemma2} we get that $v(\theta_n)=v(\theta_{n+1})=0$. Hence there is a point
$\tilde\theta_n\in (\theta_n,\theta_{n+1})$ such that $v'(\tilde\theta_n)=0$. And since
$\displaystyle\lim_{n\rightarrow\infty}\tilde\theta_n=\theta_0$, then $\displaystyle\lim_{n\rightarrow\infty}v'(\tilde\theta_n)=v'(\theta_0)=0$.
Therefore, we have a~contradiction.
\end{proof}

\begin{lemma}\label{spichak:lemma4}
If $\theta_0$ is a singular point of $W$, then it is degenerate for this field {\rm(}i.e., $w'(\theta_0)=0${\rm)}.
\end{lemma}

\begin{proof}
Since $v(\theta_0)=0$ (see Lemma \ref{spichak:lemma2}), then
$$
v(\theta)=v'(\theta_0)(\theta-\theta_0)+h(\theta), \ w(\theta)=w'(\theta_0)(\theta-\theta_0)+g(\theta),
$$
where $h$, $g$ are continuously differentiable functions and
\begin{equation} \label{spichak:eq4}
h(\theta_0)=g(\theta_0)=h'(\theta_0)=g'(\theta_0)=0.
\end{equation}
Furthermore, taking to account equation (\ref{spichak:eq1}) we have
\begin{gather*}
\left[v(\theta)\dfrac{d}{d\theta},w(\theta)\dfrac{d}{d\theta}\right]=\left[(\theta-\theta_0)(v'(\theta_0)g'(\theta)-
w'(\theta_0)h'(\theta))\right.\\
\left.\qquad+h(\theta)(w'(\theta_0)+g'(\theta))-g(\theta)(v'(\theta_0)+h'(\theta))\right]\dfrac{d}{d\theta}\\
\qquad=\left[w'(\theta_0)(\theta-\theta_0)+g(\theta)\right]\dfrac{d}{d\theta}.
 \end{gather*}
Let us divide both sides of this equality by $\theta-\theta_0$ and take the limit $\theta\rightarrow\theta_0$. Then from
relations (\ref{spichak:eq4}) and L'Hopital theorem one can easily see that $w'(\theta_0)=0$.
\end{proof}

\begin{lemma}\label{spichak:lemma5}
Under the equivalence transformation $\tilde{\theta}=f(\theta)$ singular {\rm(}resp.\ regular{\rm)} points of the vector field $W$
are mapped to singular {\rm(}resp.\ regular{\rm)} points of the vector field $\tilde W=\tilde w(\tilde\theta)\frac{d}{d\tilde\theta}$.
\end{lemma}

\begin{proof}
a) Let $w(\theta_0)\not=0$. Then there is a finite derivative $f'(\theta_0)$.
If it is not, then there is a neighborhood of the singular point $\theta_0$, where it is continuous
and $\displaystyle\lim_{\theta\rightarrow\theta_0}f'(\theta)=\pm\infty$ (the sign depends on $\deg f$). For the transformed vector field
$\tilde W$ we have the relation
$\tilde w(\tilde\theta)\frac{d}{d\tilde\theta}=w(\theta)f'(\theta)\frac{d}{d\theta}$.
Hence, $\tilde w(f(\theta))=w(\theta)f'(\theta)$. Since $w(\theta)\not=0$ in the above neighborhood, then
\begin{equation} \label{spichak:eq5}
f'(\theta)=\dfrac{\tilde w(f(\theta))}{w(\theta)},
\end{equation}
and from the continuity of the functions $f$ and $\tilde w$ it is following that the limit $\displaystyle\lim_{\theta\rightarrow\theta_0}f'(\theta)$ is finite.
Recall that a continuity of the function $\tilde w$ follows from the properties that any equivalence transformation $f$ maps $C^1$-vector fields to $C^1$-vector
fields.

Now suppose that $\tilde w(f(\theta_0))=0$ (i.e., $f(\theta_0)$ is singular). Since the right side of differential equation (\ref{spichak:eq5}) satisfies Lipschitz condition
for the argument  $f(\theta)$ uniformly on $\theta$ (function $\tilde w$ is continuously differentiable), then in a neighborhood of the point
$\theta_0$ a unique solution $f(\theta)$ exists. The constant solution in this neighborhood $f(\theta)=f(\theta_0)=\mathrm{const}$  satisfies equation
(\ref{spichak:eq5}). But the definition of the equivalence transformation contradicts to the property of one-to-one mapping.
That is $\tilde\theta_0=f(\theta_0)$ is a regular point.

b) Let $\theta_0$ be a singular point and suppose that $\tilde w(f(\theta_0))\not=0$ (i.e., $f(\theta_0)$ is regular). By Lemma~\ref{spichak:lemma3} there is a neighborhood of $\theta_0$,
in which all points are regular (except $\theta_0$). From the previous part of the proof we get (see (\ref{spichak:eq5})) for the regular points
relations $f'(\theta)>0$ ($f'(\theta)<0$) if $\deg f>0$ ($\deg f<0$). Therefore,  it is easy to show that the inverse map
$\theta=f^{-1}(\tilde\theta)$  belongs to the class of equivalence transformations. Applying the reasoning of the previous part,
we see that the regular point $\tilde\theta_0$ goes to the regular point $\theta_0$ under the mapping $f^{-1}$. So we have a contradiction.
\end{proof}

From the last Lemma it follows that the number of singular points is an invariant under any equivalence transformation.

\section{Realizations of a two-dimensional noncommutative algebra}

Taking into account Lemmas~\ref{spichak:lemma1} and~\ref{spichak:lemma3}, we suppose that there is a vector field $W$ with $n\geqslant 1$ singular points $\theta_k$.
It is easy to show that applying the composition
of equivalence transformations and rotation of the circle we achieve that
$\theta_k=\dfrac{2\pi k}{n}$, $k=0,1,\ldots,n-1$. Consider the interval
$\bigtriangleup_k=(\theta_k,\theta_{k+1})$ and denote \[ \bar\theta_k=\dfrac{\theta_k+\theta_{k+1}}{2}=\dfrac{\pi(2k+1)}{n}.\]
We construct the following continuously differentiable  transformation~for $f$ on $\bigtriangleup_k$ satisfying the conditions
\begin{equation} \label{spichak:eq6}
f(\theta_k)=\theta_k,\ \ f(\theta_{k+1})=\theta_{k+1},\ \ f(\bar\theta_k)=\bar\theta_k.
\end{equation}

Suppose that $w(\theta)>0$, $\theta\in\bigtriangleup_k$. Consider the Cauchy problem for this interval:
\begin{equation} \label{spichak:eq7}
w(\theta)f'(\theta)=1-\cos(nf(\theta)),\quad f(\bar\theta_k)=\bar\theta_k.
\end{equation}
Its solution is
\begin{equation} \label{spichak:eq8}
f(\theta)=\dfrac{2}{n}\arctan\left(-nI(\theta)\right)+\theta_k,
\quad\mbox{where}\quad
I(\theta)=\int_{\bar\theta_k}^{\theta}\dfrac{d\theta}{w(\theta)}.
\end{equation}
The integral $I(\theta)$ converges for any point of the interval $\bigtriangleup_k$.
By virtue of Lemma~\ref{spichak:lemma4} the integral diverges at the ends of this interval:
\begin{equation*}
\lim_{\theta\rightarrow\theta_k+0}I(\theta)=-\infty, \ \ \lim_{\theta\rightarrow\theta_{k+1}-0}I(\theta)=+\infty .
\end{equation*}

It is easy to show that the transformation (\ref{spichak:eq8}) satisfies conditions (\ref{spichak:eq6})
and maps the vector field $w(\theta)\frac{d}{d\theta}$ to the vector field $(1-\cos(n\tilde{\theta}))\frac{d}{d\tilde{\theta}}$.

If $w(\theta)<0$ for $\theta\in\bigtriangleup_k$ then we can analogously obtained the equivalence transformation
that maps the vector field $w(\theta)\frac{d}{d\theta}$ to $(\cos(n\tilde{\theta})-1)\frac{d}{d\tilde{\theta}}$.

Now, if we consider the vector field at the intervals $\bigtriangleup_k$, $k=0,1,\ldots,n-1$, then substituting the function
$w(\theta)=\pm(\cos(n\theta)-1)$ (omitting the tilde)  in equation (\ref{spichak:eq1}), it is easy to obtain the solution for the function
$v(\theta)$ on these specified intervals:
\begin{equation} \label{spichak:eq9}
v(\theta)=\dfrac{1}{n}\sin(n\theta)+\lambda_k(1-\cos(n\theta)),\quad  \lambda_k\in{\mathbb R}.
\end{equation}

As a result, we have the following assertion.
\begin{theorem}
Any realization of the two-dimensional noncommutative algebra of vector fields on a circle is equivalent to the form
\begin{equation} \label{spichak:eq10}
\displaystyle\left<\sigma_k(\theta)(1-\cos(n\theta))\dfrac{d}{d\theta},\
 \left(\dfrac{1}{n}\sin(n\theta)+\lambda_k(\theta)(1-\cos(n\theta))\right)\dfrac{d}{d\theta}\right>,
\end{equation}
where $\lambda_k(\theta)$ and $\sigma_k(\theta)=\pm 1$ are constants on the intervals $\displaystyle\left(\dfrac{2\pi k}{n},\dfrac{2\pi(k+1)}{n}\right)$.
\end{theorem}

Moreover, using the equivalence transformations the set of realizations  (\ref{spichak:eq10}) can be reduced to
simpler realizations where $\lambda_k(\theta)=0$. The proof of this assertion is based on the construction of
corresponding transformations on each of the intervals $\displaystyle\left(\dfrac{2\pi k}{n},\dfrac{2\pi(k+1)}{n}\right)$. So we have
\begin{theorem}
Any realization of the two-dimensional noncommutative algebra of vector fields on a circle is equivalent to the form
\begin{equation} \label{spichak:eq11}
\displaystyle\left<\sigma_k(\theta)(1-\cos(n\theta))\dfrac{d}{d\theta},\
\dfrac{1}{n}\sin(n\theta)\dfrac{d}{d\theta}\right>,
\end{equation}
where $\sigma_k(\theta)=\pm 1$ on the intervals $\displaystyle\left(\dfrac{2\pi k}{n},\dfrac{2\pi(k+1)}{n}\right)$.
\end{theorem}

\section{Realizations of three-dimensional algebra $\mathrm{sl}(2,\mathbb R)$}
The commutation relations of the algebra $\mathrm{sl}(2,\mathbb R)=\left<W,V,Z\right>$ are
\begin{equation} \label{spichak:eq12}
[V,W]=W,\ [Z,V]=Z,\ [W,Z]=-2V,
\end{equation}
where $W=w(\theta)\frac{d}{d\theta}$, $V=v(\theta)\frac{d}{d\theta}$, $Z=z(\theta)\frac{d}{d\theta}$.\\
Using relations (\ref{spichak:eq12}) one can write out the system of differential equations for the functions
$w(\theta)$, $v(\theta)$, $z(\theta)$. It is not difficult to obtain the general solution of it.
Thus, we have the following result
\begin{theorem}
All realizations of the algebra $\mathrm{sl}(2,\mathbb R)$ up to the equivalence transformation
have one of the following form
\begin{equation}\label{spichak:eq13}
\begin{array}{l}
\displaystyle \left<W,V,Z\right>=\left<\sigma_k(\theta)(1-\cos(n\theta))\dfrac{d}{d\theta},\
\dfrac{1}{n}\sin(n\theta)\dfrac{d}{d\theta}\right.\\[2mm]
\displaystyle\hspace{2.5cm} \left.\dfrac{1}{n^2}\sigma_k(\theta)(1+\cos(n\theta))\dfrac{d}{d\theta}\right>
\end{array}
\end{equation}
where $\sigma_k(\theta)=\pm 1$ are constants on the intervals
$\displaystyle\left(\dfrac{2\pi k}{n},\dfrac{2\pi(k+1)}{n}\right)$.
\end{theorem}

So, all realizations two- and three-dimensional algebras (\ref{spichak:eq1}) and
(\ref{spichak:eq12}), respectively, are defined by a set of n signs,
which we call the {\it signature} $\pi$ of these realizations. So far we have considered
equivalence transformations when $f(0)=0$ with degree of map $\deg f=1$.
If, in addition to such transformations, we consider rotations of a circle ($f(0)\ne 0$),
as well as reflection ($\deg f=-1$) then some of such realizations with different signature
$\pi$ and $\tilde\pi$ can be equivalent.

\section{Number of realizations. Some combinatorial formulas}
{\bf A. Rotations.} Let us consider the transformation-rotation of a circle $O_k$, $k\in {\mathbb Z}$ with angle
$\dfrac{2\pi}{n}$.
Let's call $\sigma=(\sigma_0,\sigma_1,\ldots,\sigma_{n-1})$ as signature with length $n$ of realizations
(\ref{spichak:eq11}),  (\ref{spichak:eq13}). Then clockwise rotation
$O_k$ of the circle leads to substitution of signature: $O_k(\sigma)=\tilde\sigma=(\tilde\sigma_0,\tilde\sigma_1,\ldots,\tilde\sigma_{n-1})$, where
$\tilde\sigma_l=\sigma_{n-k+l(\mod n)}$, $l=0,1,\ldots,n-1$.

Signatures $\sigma$ and $\tilde\sigma$ we called equivalent relatively rotation $O_k$.
The hole set of signatures is divided into equivalence classes. Let us call signature $\sigma$ as
{\it $d$-periodic} if $d\in {\mathbb Z}$, $1\leq d\leq n$, d is minimal such that $O_d(\sigma)=\sigma$.

It is not difficult to show:\\
1) that $d$ is divider of $n$: $d|n$;\\
2) if $d\neq d'$, then $d$- and $d'$-periodic signatures are not equivalent.

Let we denote:\\
-- ${\it L}_n$ as number of equivalence classes;\\
-- ${\it M}_d$ as number of equivalence classes of $d$-periodic signatures with length $n$;\\
-- ${\it N}_d$ as  number of different (maybe equivalent) $d$-periodic signatures with length $d$.

Then, obviously, ${\it N}_d=d{\it M}_d$. We have
\begin{equation} \label{spichak:eq14}
{\it L_n}=\sum_{d|n}{\it M}_d= \sum_{d|n}\dfrac{{\it N}_d}{d}.
\end{equation}
Because the hole number of different signatures is equally $2^n$ then
\begin{equation} \label{spichak:eq15}
\sum_{d|n}{\it N}_d=2^n,\  {\it N}_1=2.
\end{equation}
So, we have recurrence formula (\ref{spichak:eq15}) for calculation of ${\it N}_d$,
and then for number of equivalence classes (\ref{spichak:eq14}).

For example, if $n=p$ is
a prime number, then ${\it N}_1+{\it N}_p=2^p$, i.e. ${\it N}_p=2^p-2$.

If $n=p_1p_2$, $p_i$ are prime numbers, then
${\it N}_1+{\it N}_{p_1}+{\it N}_{p_2}+{\it N}_{p_1p_2}=2^{p_1p_2}$, i.e.
${\it N}_n=2^n-2^{p_1}-2^{p_2}+2$.\\[2mm]
${\it L_n}=\dfrac{{\it N}_1}{1}+\dfrac{{\it N}_{p_1}}{p_1}+\dfrac{{\it N}_{p_2}}{p_2}+\dfrac{{\it N}_{p_1p_2}}{p_1p_2}=
\vspace{3mm}\dfrac{2}{1}+\dfrac{2^{p_1}-2}{p_1}+\dfrac{2^{p_2}-2}{p_2}+\dfrac{2^n-2^{p_1}-2^{p_2}+2}{p_1p_2}$.

Let's $d=p_1^{\alpha_1}p_2^{\alpha_2}\ldots p_s^{\alpha_s}$. We introduce the following notation:
$d_i=\dfrac{d}{p_i}$, $d_{ij}=\dfrac{d}{p_ip_j}$, $i,j=1,\ldots ,s$.
Using the induction method we can proof the following
\begin{theorem}
Suppose that $n=p_1^{\alpha_1}p_2^{\alpha_2}\ldots p_s^{\alpha_s}$, where $p_i$ are prime numbers,
$\alpha_i\in {\mathbb Z}$. Then
$$
{\it N}_d=2^d-\sum_i2^{d_i}+\sum_{i< j}2^{d_{ij}} .
$$
\end{theorem}

{\bf B. Reflection.} Let us consider the transformation-reflection ${\it T}$ on "vertical" axe,
that is such a transformation from the degree of mapping $\deg {\it T}=-1$, that is
$\theta\rightarrow\theta'= 2\pi-\theta$, $0\leq\theta< 2\pi$.
It's easy to show that under transformation  ${\it T}$ a realization with the signatura $\sigma$ transforms
to a realization with a signatura
$T(\sigma)=\tilde\sigma=(-\sigma_{n-1},-\sigma_{n-2},\ldots,-\sigma_{0})$.
It is not difficult to show that $O_kT=TO_{-k}$. Also, for realizations (\ref{spichak:eq11}),  (\ref{spichak:eq13})
any equivalence transformation is composition: $f=f_00_kT$, where $f_0(0)=0$.

Let denote ${\it K}_n$ as number of equivalence classes respectively equivalence transformation
$0_k$ and $T$, $k\in {\mathbb Z}$. Then it is not difficult to proof the next
\begin{theorem}
1) If n is odd number then ${\it  K}_n=\dfrac{{\it L}_n}{2}$.\\
2) If n is even number then ${\it K}_n=\dfrac{{\it L}_n+{\it I}_n}{2}$;\\
where ${\it I}_n$ is a number of equivalence classes (with respect to transformations $O_k$) with signatures having the same number of $+$ and $-$, and which are invariant under map $T$. If $n=2k$, then ${\it I}_n=2^{k-1}$.
\end{theorem}

%{\small \bibliographystyle{plain}

%\bibliography{Spichak}}

\begin{thebibliography}{100}

\bibitem{spichak:Dubrovin}
Dubrovin B.A., Novikov S.P. and Fomenko A.T., {\it Modern geometry: methods and applications},
Nauka, Moscow, 1986 (in Russian).

\bibitem{spichak:Gonzalez}
Gonz\'alez-L\'opez A., Kamran N. and Olver P.J., Lie algebras of vector fields in the real plane,
{\it Proc. London Math. Soc.}  {\bf 64} (1992), 339--368.

%\bibitem{spichak:Lie1}
%Lie S., Theorie der Transformationsgruppen I, {\it Math. Ann.} {\bf 16}
%(1880), 441--528.

%\bibitem{spichak:Lie1a}
%Lie S., Gesammetle Abhandlungen, (1927), 1--94 (in German).

\bibitem{spichak:Lie2}
Lie S., {\it Theorie der Transformationsgruppen}, Vol. 3, Teubner, Leipzig, 1893.

\bibitem{spichak:Popovych}
Popovych R.O., Boyko V.M., Nesterenko M.O. and Lutfullin M.W.,
Realizations of real low-dimensional Lie algebras, {\it J. Phys. A: Math. Gen.}
{\bf  36} (2003),  7337--7360; arXiv:math-ph/0301029.

\bibitem{spichak:Presli}
Pressley A. and Segal G., {\it Loop groups}, Mir, Moscow, 1990  (in Russian).

\bibitem{spichak:Sergeev}
Sergeev A.G., {\it Geometric quantization of loop spaces},  Modern. Math. Probl., Vol.~13, Steklov Mathematical Institute of RAS, Moscow, 2009 (in Russian).

\bibitem{spichak:Stepanov}
Stepanov V.V.,  {\it Course of differential equations},  Fizmatlit, Moscow, 1950 (in Russian).

\bibitem{spichak:Strigunova}
Strigunova M.S., Finite-dimensional subalgebras in the Lie algebra of vector fields on a circle, {\it Tr. Mat. Inst. Steklova} {\bf 236} (2002), Differ. Uravn. i Din. Sist., 338--342 (in Russian); translation in {\it Proc. Steklov Inst. Math.} {\bf 236} (2002), 325--329.

\bibitem{spichak:Yehorchenko}
Yehorchenko I.A., Nonlinear representation of the Poincar\'e algebra and invariant equations, {\it Symmetry Analysis of Equations of Mathematical Physics}, 62--66, Inst. of Math. of NAS of Ukraine, Kyiv,  1992.

\bibitem{spichak:Zaitseva}
Zaitseva V.A., Kruglov V.V., Sergeev A.G., Strigunova M.S. and Trushkin K.A.,
Remarks on the loop groups of compact Lie groups and the diffeomorphism group of the circle,
{\it Tr. Mat. Inst. Steklova} {\bf 224} (1999), Algebra. Topol. Differ. Uravn. i ikh Prilozh., 139--151; translation in {\it Proc. Steklov Inst. Math.} {\bf 224} (1999), 124--136.

\bibitem{spichak:Zhdanov2}
Zhdanov R.Z, Lahno V.I. and Fushchych W.I., On covariant realizations of the Euclid group,
{\it Comm. Math. Phys.} {\bf 212} (2000), 535--556.

\end{thebibliography}

\end{document}